\documentclass[english]{amsart}
\usepackage[utf8]{inputenc}
\usepackage[english]{babel}
\usepackage{amsthm,amsmath,amssymb,amsfonts,amscd,epsfig}
\usepackage{graphicx,enumerate,csquotes}
\usepackage[all]{xy}
\usepackage{booktabs,delarray,mathtools,float,faktor,xfrac}
\usepackage{textgreek,verbatim,subfigure,tabularx,longtable,multirow,spalign}
\usepackage{footnote}\makesavenoteenv{tabular}\makesavenoteenv{table}
\usepackage{pgfplots}
\usepackage{tikz}
\usetikzlibrary{matrix}
\usetikzlibrary{cd,positioning,arrows}
\usepackage{mathrsfs}

\newtheorem{thm}{Theorem}[section]

\newtheorem{lem}[thm]{Lemma}
\newtheorem{prop}[thm]{Proposition}
\theoremstyle{definition}
\newtheorem{defn}[thm]{Definition}
\theoremstyle{remark}
\newtheorem{rem}[thm]{Remark}
\numberwithin{equation}{section}
\newtheorem{examp}[thm]{Example}

\renewcommand{\epsilon}{\varepsilon}

\newcommand{\cC}{\mathcal C}

\newcommand{\bbR}{\mathbb R}
\newcommand{\bbT}{\mathbb T}

\newcommand{\bbZ}{\mathbb Z}

\newcommand{\mfg}{\mathfrak g}
\newcommand{\mfh}{\mathfrak h}

\newcommand{\bd}{\partial}

\newcommand{\frbd}[1]{\frac{\bd}{\bd #1}}

\DeclareMathOperator{\rank}{rank}
\pgfplotsset{compat=1.14}

\begin{document}

\title[Rigidity of cotangent lifts and integrable systems]{Rigidity of cotangent lifts and integrable systems} 

\author{Pau Mir}
\address{ Pau Mir,
Laboratory of Geometry and Dynamical Systems, Department of Mathematics, EPSEB, Universitat Polit\`{e}cnica de Catalunya BGSMath Barcelona Graduate School of
Mathematics in Barcelona}
\email{pau.mir.garcia@upc.edu}
\author{Eva Miranda}\address{ Eva Miranda,
Laboratory of Geometry and Dynamical Systems, Department of Mathematics, EPSEB, Universitat Polit\`{e}cnica de Catalunya BGSMath Barcelona Graduate School of
Mathematics in Barcelona and
\\ IMCCE, CNRS-UMR8028, Observatoire de Paris, PSL University, Sorbonne
Universit\'{e}, 77 Avenue Denfert-Rochereau,
75014 Paris, France
 }\thanks{  {Eva} Miranda  is supported by the Catalan Institution for Research and Advanced Studies via an ICREA Academia Prize 2016 and partially supported  by the grant reference number MTM2015-69135-P (MINECO/FEDER). Pau Mir and Eva Miranda  are both  partially supported   by the grant reference number {2017SGR932} (AGAUR). Pau Mir  is supported by an INIREC grant (beca d'iniciació a la recerca) supported by Eva Miranda's ICREA Academia grant. }
 \email{ eva.miranda@upc.edu}
\begin{abstract}
    In this article we generalize a theorem by Palais on the rigidity of compact group actions to cotangent lifts. We use this result to prove rigidity results for integrable systems on symplectic manifolds including sytems with degenerate singularities which are invariant under a torus action.\end{abstract}

\maketitle

\section{Introduction}

In \cite{PalaisR61} Palais proved that two close compact Lie group actions on a compact manifold are equivalent in the sense that there exists a diffeomorphism conjugating both actions.
This rigidity result comes hand-in-hand with other classical stability results à la Mather-Thom for differentiable maps in the 60's and 70's for which stability yields equivalence of close maps (see for instance \cite{mather} and \cite{thom}). 

Symplectic manifolds provide a natural landscape to test stability ideas as among the classical actions of Lie groups on symplectic manifolds the ones admitting a \emph{moment map} stand out. These are Hamiltonian actions where the group action can be read off from a mapping   $\mu:M\longrightarrow \mathfrak{g}^*$ where $ \mathfrak{g}$ is the Lie algebra of the Lie group.

In \cite{MirandaRSME} it was proved that $C^2$-close symplectic actions on a compact symplectic manifold are equivalent in the sense that not only the actions are conjugated by a diffeomorphism but this diffeomorphism preserves the symplectic form. The proof in \cite{MirandaRSME} (see also \cite{MirandaMonnierZung}) uses the path method requiring differentiability of degree 2  as the diffeomorphism yielding the equivalence comes from integration of a time-dependent vector field. Generalizations of this result can be easily achieved for Hamiltonian actions in the symplectic context. In the more general Poisson context technical complications occur due to the lack of a general path method in Poisson geometry and fine Nash-Moser techniques come to the rescue to prove rigidity of Hamiltonian actions of semisimple Lie groups of compact type on Poisson manifolds as proven in  \cite{MirandaMonnierZung}.
Those results can be obtained either globally (for compact manifolds) or semilocally (in the neighbourhood of a compact submanifold which is invariant  by the group actions).

The lift of Lie group actions to the cotangent bundle,  naturally equipped with a canonical symplectic form, provide natural examples of Hamiltonian actions. Non-compactness of cotangent bundles leaves the study of equivalence of actions out of the radar of the case needs to be re-examined with fresh eyes.
In this note we analyze the case of cotangent lifted actions  where we can easily prove the equivalence of Hamiltonian actions on non-compact manifolds (cotangent bundles) by lifting the diffeomorphism of Palais from the base. This simple idea allows to reduce the required degree of differentiability by 1 from the case of compact group actions on compact symplectic manifolds.
We present a new result on rigidity of lifted actions, which can be thought as an extension of Palais rigidity Theorem to the cotangent lift of an action of a compact group. It has the advantage of being useful at the level of the cotangent bundle, which is a non-compact manifold, in contrast with the compactness required for the manifold in the original Palais Theorem.

Cotangent lifted actions may, a priori, seem a small class of actions to consider, however this class includes the wide class of regular integrable systems as we proved in \cite{KiesenhoferMiranda}.  The action-angle coordinate theorem for integrable systems can be rephrased (see \cite{KiesenhoferMiranda}) as follows: any integrable system is equivalent in a neighbourhood of a Liouville torus to the integrable system given by the cotangent lift of translations of this torus to $T^*(\mathbb T^n)$.

In this sense group actions turn out to be a useful tools to understand integrable systems.
But what happens with singularities of integrable systems? As a consequence of the cotangent lift result above and  rigidity theorem for cotangent lifts, it follows that integrable systems whose singularities are only of regular and of elliptic type are rigid inside the integrable class.
Some of these results can be reproved using normal form theorems for the integrable system and the symplectic forms. However, our technique reveals to be useful also when there are no normal forms known for degenerate singularities invariant by circle actions. We end up this article by proving a rigidity result for a class of degenerate singularities of  integrable systems.

\textbf{Organization of this article:} In Section \ref{sec:PrelimsRigidity} we give the definitions of closeness and rigidity and we present some known results on rigidity. We recall the definition of the cotangent lift of a Lie group action and  we briefly describe the main results on semilocal classification of non-degenerate singularities of integrable Hamiltonian systems. In Section \ref{sec:RigidityCotLift} we state and prove Theorem \ref{thm:palaiscotangentlift}, a result on rigidity of the cotangent lift of a close action of a compact Lie group on a compact manifold. Finally, in Sections \ref{sec:ApplIntSys} and \ref{sec:ApplDegSing} we give applications of theorem \ref{thm:palaiscotangentlift} to integrable Hamiltonian systems which have non-degenerate singularities (see Theorem \ref{thm:rigiditysystemwithdegensing}) for which stability and infinitesimal stability à la Mather hold (see remark \ref{rem:mather}) and a class of degenerate singularities for which no normal form theorem is known. 

\section{Preliminaries }
\label{sec:PrelimsRigidity}
In this article manifolds and maps are assumed to be $\mathcal{C}^\infty$ unless otherwise stated. The notation  follows \cite{GuillGinzKarsh} and \cite{MirandaCentrEur}.
\subsection{Rigidity theorems}
 Following Palais in \cite{PalaisR60}, we recall the definition of $C^k$-close actions.

\begin{defn}
Let $f,g:M\longrightarrow N$ be two smooth maps between smooth manifolds of dimension $m$ and $n$, respectively. Suppose that $(x_1,\dots,x_m)$ is a coordinate system for $K\subset M$ compact and $(y_1,\dots,y_n)$ is a coordinate system for $V\subset N$. Suppose that $f(K)\subset V$ and $g(K)\subset V$. Then, $f$ and $g$ are \textit{$C^k$-close maps}, for $k\geq0$, if there exists an $\epsilon>0$ such that $|y_i\circ f(p)-y_i\circ g(p)|<\epsilon$ for $p\in K$ and $i=1,\dots,n$ and $$\Big|\frac{\bd^r(y_i\circ f)}{\bd x_{j1}\cdots\bd x_{jr}}(p) - \frac{\bd^r(y_i\circ f)}{\bd x_{j1}\cdots\bd x_{jr}}(p)\Big| <\epsilon,$$ for $p\in K$, $r\leq k$, $i=1,\dots,n$ and $j_{\alpha}=1,\dots,m$.
\label{def:closemap}
\end{defn}

For Lie group actions the definition of closeness is the natural one, considering that the source space is the product of two smooth manifolds, hence a smooth manifold. We recall now the definition of rigidity of a group action.

\begin{defn}[Rigid action]
Let a Lie group $G$ act smoothly on a manifold $M$ and let $\rho:G \times M\longrightarrow M$ denote this action. The action $\rho$ is \textit{rigid} if for every smooth one-parameter family of actions $\rho_t$ of $G$ on $M$ there exists a one-parameter family of diffeomorphisms $h_t:M\longrightarrow M$ which conjugate $\rho$ to $\rho_t$ for all $t$ in a small interval $(-\epsilon,\epsilon)\subset \mathbb R$.
\label{def:rigidaction}
\end{defn}

Richard Palais already proved in \cite{PalaisR61} an important rigidity result, the existence of a diffeomorphism that conjugates $C^1$-close actions of a compact Lie group on a compact manifold.

\begin{thm}[Palais]
Let $G$ be a compact Lie group and $M$ a compact manifold. Let $\rho_1,\rho_2: G \times M \longrightarrow M$ be two actions which are $C^1$-close. Then, there exists a diffeomorphism $\varphi$ of class $C^1$ that conjugates $\rho_1$ and $\rho_2$, making them equivalent. This diffeomorphism belongs to the arc-connected component of the identity.
\label{th:palais}
\end{thm}

In the case of the manifold being symplectic, Palais Theorem was extended to the following Theorem to obtain that the diffeomorphism conjugating the two close symplectic actions is a symplectomorphism. This was proved by Miranda (see \cite{MirandaRSME} or \cite{MirandaMonnierZung}).

\begin{thm}[Miranda]
Let $G$ be a compact Lie group and $(M,\omega)$ a compact symplectic manifold. Let $\rho_1,\rho_2: G \times M \longrightarrow M$ be two symplectic actions which are $C^2$-close. Then, there exists a symplectomorphism $\varphi$ that conjugates $\rho_1$ and $\rho_2$, making them equivalent.
\label{th:palaissymplectic}
\end{thm}

In the proof of Theorem \ref{th:palaissymplectic}, the diffeomorphism given by Palais Theorem is used, together with the Moser path method and a De Rham homotopy operator, to prove that the symplectic structure is equivariantly preserved.

\subsection{The cotangent lift of a group action}
\label{sec:cotangentlift}
The cotangent bundle of a smooth manifold can be naturally equipped with a symplectic structure in the following way. Let $M$ be a differential manifold and consider its cotangent bundle $T^{*}M$. There is an intrinsic canonical linear form $\lambda$ on $T^{*}M$ defined pointwise by
\begin{equation*}
    \langle \lambda_p,v\rangle=\langle p,d\pi_pv\rangle, \hspace{25pt} p=(m,\xi)\in T^{*}M,v\in T_p(T^{*}M),
\end{equation*}
where $d\pi_p:T_p(T^{*}M)\longrightarrow T_mM$ is the differential of the canonical projection at $p$. In local coordinates $(q_i,p_i)$, the form is written as $\lambda=\sum_i p_i\,dq_i$ and is called the \textit{Liouville 1-form}. Its differential $\omega=d\lambda=\sum_i dp _i\wedge dq_i$ is a symplectic form on $T^{*}M$.

\begin{defn}
Let $\rho:G\times M \longrightarrow M$ be a group action of a Lie group $G$ on a smooth manifold $M$. For each $g\in G$, there is an induced diffeomorphism $\rho_g:M\longrightarrow M$. The \textit{cotangent lift of $\rho_g$}, denoted by $\hat{\rho}_g$, is the diffeomorphism on $T^{*}M$ given by
\begin{equation*}
    \hat{\rho}_g(q,p):=(\rho_g(q),((d{\rho_g)}_q^{*})^{-1}(p)),\hspace{25pt}(q,p)\in T^*M
\end{equation*}
which makes the following diagram commute:
\begin{center}
\begin{tikzpicture}
  \matrix (m) [matrix of math nodes,row sep=4em,column sep=4em,minimum width=2em]
  {T^*M & T^*M \\
   M & M\\};
  \path[-stealth]
    (m-1-1) edge node [right] {$\pi$} (m-2-1)
            edge node [above] {$\hat{\rho}_g$} (m-1-2)
    (m-2-1) edge node [above] {${\rho}_g$} (m-2-2)
    (m-1-2) edge node [right] {$\pi$} (m-2-2);
\end{tikzpicture}
\end{center}
\end{defn}

\begin{lem}
The induced diffeomorphism $\hat{\rho}_g$ preserves the form $\lambda$ and, hence, preserves the symplectic form $\omega$.
\label{lem:cotangentliftpres1form}
\end{lem}

\begin{proof}
We will prove that, in general, that given a diffeomorphism $\rho:M\longrightarrow M$, its cotangent lift preserves the canonical form $\lambda$. At a point $p=(m,\xi)\in T^*M$, we have:
\begin{align*}
    \lambda_p&=(d\pi)^*_p\xi=\\
    &=(d\pi)^*_p(d\rho)^*_{m}\left((d\rho)^*_{m}\right)^{-1}\xi=\\
    &=(d(\rho\circ\pi))^*_{p}\left((d\rho)^*_{m}\right)^{-1}\xi=\\
    &=(d(\pi\circ\hat\rho))^*_{p}\left((d\rho)^*_{m}\right)^{-1}\xi=\\
    &=(d\hat\rho)^*_{p}(d\pi)^*_{\hat{\rho}(p)}\left((d\rho)^*_{m}\right)^{-1}\xi=\\
    &=(d\hat\rho)^*_{p}\lambda_{\hat\rho(p)},
\end{align*}
where we used the definitions of the Liouville $1$-form and the cotangent lift and the fact that $\rho\circ\pi=\pi\circ\hat\rho$. Then, the canonical $1$-form is preserved by $\hat\rho$.

As a consequence:
$$\hat\rho^*(\omega)=\hat\rho^*(d\lambda)=d(\hat\rho^*\lambda)=d\lambda=\omega.$$

So, the cotangent lift $\hat{\rho}_g$ preserves the Liouville form and the symplectic form of $T^*M$.
\end{proof}

The following two examples contain the explicit computations and expressions of simple cotangent lifts which indeed give rise to hyperbolic and focus-focus pieces respectively of an integrable system with non-degenerate singularities (as we will see in the next subsection).

\begin{examp}
Consider the action of $(\bbR,+)$ on $\bbR$ given by:
$$\begin{array}{rccc}
  \rho:& \bbR\times\bbR &\longrightarrow &\bbR\\
  &\left(t,q\right) &\longmapsto &e^{-t} q
\end{array}$$
and the induced an action $\rho_t:\bbR \longrightarrow \bbR$. The differential of $\rho_t$ at a point $q\in\bbR$ is:

$$\begin{array}{rccc}
  (d\rho_t)_q: & T_q\bbR & \longrightarrow & T_q\bbR\\
  & p & \longmapsto & e^{-t} p
\end{array}$$


Then, $((d\rho_t)_q^*)^{-1}$ acts as $p \longmapsto e^{t}p$, and the cotangent lift $\hat{\rho}_t$ associated to the group action $\rho_t$, in coordinates $(q,p)$ of $T^{*}\bbR$ is exactly:
$$\begin{array}{rccc}
\hat{\rho}: & T^{*}\bbR &\longrightarrow & T^{*}\bbR\\
    & \begin{pmatrix}q\\p\end{pmatrix}&\longmapsto &
    \begin{pmatrix}e^{-t} q\\e^t p \end{pmatrix}
\end{array}$$
\label{examp:hyperboliccotangentlift}
\end{examp}

\begin{examp}
Consider the action of a rotation and a radial dilation on $\mathbb{R}^2$ given by:
$$\begin{array}{rcccl}
  \rho : &(S^1\times\mathbb{R})\times\mathbb{R}^2 &\longrightarrow & \mathbb{R}^2 &\\
  &\large((\theta,t),\begin{pmatrix}
            x_1\\
            x_2
            \end{pmatrix}\large)
            & \longmapsto &
            \rho_{\theta,t}\begin{pmatrix}
            x_1\\
            x_2
            \end{pmatrix}
            &=e^{-t}
            \begin{pmatrix}
            \cos{\theta} & \sin{\theta}\\
            -\sin{\theta} & \cos{\theta}
            \end{pmatrix}
            \begin{pmatrix}
            x_1\\
            x_2
            \end{pmatrix}
\end{array}$$

The differential of the induced action $\rho_{\theta,t}$ at a point $x=(x_1,x_2)$ is the following linear map:

$$\begin{array}{rccc}
d\rho_{\theta,t} : &T_{x}\mathbb{R}^2 &\longrightarrow & T_{x}\mathbb{R}^2\\
    &\begin{pmatrix}
            y_1\\
            y_2
            \end{pmatrix}           &\longmapsto&
            e^{-t}\begin{pmatrix}
            y_1\cos{\theta}+y_2\sin{\theta}\\
            -y_1\sin{\theta}+y_2\cos{\theta}
            \end{pmatrix}
    \label{eq:liftff}
\end{array}$$

Then, $((d\rho_{\theta,t})^*)^{-1}$ acts as:
$$         \begin{pmatrix}
            y_1\\
            y_2
            \end{pmatrix}\longmapsto
            e^t\begin{pmatrix}
            \cos\theta & \sin\theta\\
            -\sin\theta & \cos\theta
            \end{pmatrix}
            \begin{pmatrix}
            y_1\\
            y_2
            \end{pmatrix}$$

And the cotangent lift $\hat{\rho}_{\theta,t}$ associated to the group action is:
$$\begin{array}{rccc}
\hat{\rho}_{\theta,t} :& T^{*}\mathbb{R}^2 &\longrightarrow & T^{*}\mathbb{R}^2\\
    &\begin{pmatrix}
            x_1\\
            x_2\\
            y_1\\
            y_2
            \end{pmatrix}           &\longmapsto&
            \begin{pmatrix}
            e^{-t}(x_1\cos{\theta}+x_2\sin{\theta})\\
            e^{-t}(-x_1\sin{\theta}+x_2\cos{\theta})\\
            e^t(y_1\cos{\theta}+y_2\sin{\theta})\\
            e^t(-y_1\sin{\theta}+y_2\cos{\theta})
            \end{pmatrix}
    \label{eq:diff2}
\end{array}$$
\label{examp:focusfocuscotangentlift}
\end{examp}

The cotangent lift of a Lie group $G$ on a manifold $M$, which is an action on $(T^*M,\omega_{T^*M})$, is automatically  Hamiltonian (see for instance \cite{GuilleminSternbergSTP}). This makes the cotangent lift a natural and powerful tool for the formulation of integrable systems, specially in the context of mechanics.

\subsection{Semilocal description of non-degenerate singularities in integrable Hamiltonian systems}
\label{sec:SingIntSys}

A Hamiltonian system is completely integrable if it is defined by $n$ first integrals in involution with respect to the Poisson bracket. Completely integrable Hamiltonian systems are closely related to Lagrangian foliations through the following result.

\begin{prop}
Let $f_1, \dots,f_n$ be $n$ functions such that $\{f_i,f_j\}=0,\forall i,j$. Suppose that $d_p f_1\wedge\dots \wedge d_pf_n\neq 0$ at a point $p\in M$. Then, the distribution generated by the Hamiltonian vector fields $\mathcal D=\langle X_{f_1},\dots, X_{f_n} \rangle$ is involutive and the leaf through $p$ is a Lagrangian submanifold.
\label{prop:LagrFolHamIntSys}
\end{prop}


The dynamics of an integrable system $F=(f_1,\dots,f_n)$ is explained by the Arnold-Liouville-Mineur Theorem at the regular points, namely, at the points of the manifold where the differential $dF=(df_1, \dots,df_n)$ is not singular.

\begin{thm}[Arnold-Liouville-Mineur]
Let $(M^{2n},\omega)$ be a symplectic manifold. Let $f_1,\ldots,f_n$ functions on $M$ which are functionally independent (i.e. $df_1\wedge\dots\wedge df_n\neq 0$) on a dense set and which are pairwise in involution. Assume that $m$ is a regular point of $F=(f_1,\ldots,f_n)$ and that the level set of $F$ through $m$, which we denote by $\mathcal F_m$, is compact and connected.

Then, $\mathcal F_m$ is a torus and on a neighbourhood $U$ of $\mathcal F_m$ there exist ${\bbR}$-valued smooth functions $(p_1,\dots,p_n)$ and ${\bbR}/{\bbZ}$-valued smooth functions $({\theta_1},\dots,{\theta_n})$ such that:
\begin{enumerate}
    \item The functions $(\theta_1,\dots,\theta_n,p_1,\dots,p_n )$ define a diffeomorphism $U\simeq\bbT^n\times B^{n}$.
    \item The symplectic structure can be written in terms of these coordinates as
    \begin{equation*}
       \omega=\sum_{i=1}^n d \theta_i \wedge dp_i.
    \end{equation*}
    \item The leaves of the surjective submersion $F=(f_1,\dots,f_{s})$ are given by the projection onto the
      second component $\bbT^n \times B^{n}$, in particular, the functions $f_1,\dots,f_s$ depend only on
      $p_1,\dots,p_n$.
\end{enumerate}
The coordinates $p_i$ are called action coordinates; the coordinates $\theta_i$ are called angle coordinates.
\label{thm:ALM}
\end{thm}

The Arnold-Liouville-Mineur Theorem was restated by Kiesenhofer and Miranda in \cite{KiesenhoferMiranda} revealing that at a semilocal level the regular leaves are equivalent to a completely toric cotangent lift model.

\begin{thm}
Let $F=(f_1,\dots,f_n)$ be an integrable system on a symplectic manifold $(M,\omega)$. Then, semilocally around a regular Liouville torus, the system is equivalent to the cotangent model $(T^* \bbT^n)_{can}$ restricted to a neighbourhood of the zero section $(T^* \bbT^n)_0$ of $T^* \bbT^n$.
\label{thm:cotangentliftALM}
\end{thm}

At the singular points, the degeneracy of $dF$ determines in general how difficult is to understand the dynamics, and for the case of non-degenerate singular points there are powerful results. The following definitions give the precise details of these concepts.

\begin{defn}
A point $p\in M^{2n}$ is a \textit{singular point} of an integrable Hamiltonian system given by $F=(f_1, \dots,f_n)$ if the rank of $dF=(df_1, \dots,df_n)$ at $p$ is less than $n$. The singular point $p$ has \textit{rank} $k$ and \textit{corank} of $n-k$ if $\rank(dF)_p=\rank\left((df_1)_p,\dots,(df_n)_p\right)=k$.
\label{def:rankcoranksingpoint}
\end{defn}

\begin{defn}
Let $\mfg$ be a Lie algebra. A \textit{Cartan subalgebra} $\mfh$ is a nilpotent subalgebra of $\mfg$ that is self-normalizing, i.e., if $[X,Y]\in \mfh$ for all $X\in \mfh$, then $Y\in \mfh$. If $\mfg$ is finite-dimensional and semisimple over an algebraically closed field of characteristic zero, a Cartan subalgebra is a maximal abelian subalgebra (a subalgebra consisting of semisimple elements).
\label{def:Cartansubalgebra}
\end{defn}

\begin{defn}
Let $(M^{2n},\omega)$ be a symplectic manifold with an integrable Hamiltonian system of $n$ independent and commuting first integrals $f_1,\dots,f_n$. Consider a singular point $p\in M$ of rank $0$, i.e. $(df_i)_p=0$ for all $i$. It is called a \textit{non-degenerate singular point} if the operators $\omega^{-1}d^2f_1,\dots,\omega^{-1}d^2f_n$ form a Cartan subalgebra in the symplectic Lie algebra $\mathfrak{sp}(2n,\bbR)=\mathfrak{sp}(T_p M,\omega)$.
\label{def:non-degsingularpoint}
\end{defn}

\begin{rem}
The operators $\omega^{-1}d^2f_i$, where $df_i$ is the Hessian of $f_i$, associate a function to the Hessian by visualizing the Hessian as a quadratic form $H(u,v)$ and taking the symplectic dual of the function obtained. A good reference for details of the algebraic construction of the Cartan subalgebra is \cite{BolsinovFomenko}.
\end{rem}

The classification of non-degenerate critical points of the moment map in the real case was obtained by Williamson \cite{Williamson}. In the complex case, all the Cartan subalgebras are conjugate and hence there is only one model for non-degenerate critical points of the moment map.

\begin{thm}[Williamson]
For any Cartan subalgebra $\cC$ of $\mathfrak{sp}(2n,\bbR)$, there exists a symplectic system of coordinates $(x_1,\dots, x_n, y_1,\dots,y_n)$ in $\bbR^{2n}$ and a basis $f_1,\dots,f_n$ of $\cC$ such that each of the quadratic polynomials $f_i$ is one of the following:

\begin{align*}
&\quad f_i = x_i^2 + y_i^2 & {\rm for}& \ \ 1 \leq i \leq k_e \\
&\quad f_i = x_iy_i & {\rm for} &\ \ k_e+1 \leq i \leq k_e+k_h \\
&\begin{cases}f_i = x_i y_{i+1}- x_{i+1} y_i \\
f_{i+1} = x_i y_i + x_{i+1} y_{i+1}\end{cases}
& {\rm for} &\ \ i = k_e+k_h+ 2j-1, \ 1 \leq j \leq k_f
\label{eq:williamsonbasis}
\end{align*}

The three types are called elliptic, hyperbolic and focus-focus, respectively.
\label{thm:WilliamsonCartanSubalgebra}
\end{thm}

\begin{rem}
Notice that the focus-focus components always go by pairs. Because of theorem \ref{thm:WilliamsonCartanSubalgebra}, the triple $(k_e,k_h,k_f)$ at a singular point it is an invariant. It is also an invariant of the orbit of the integrable system through the point \cite{Zung-AL1996}.
\end{rem}

If $p$ is a non-degenerate singularity of the moment map $F$, it is characterized by four integer numbers, the rank $k$ of the singularity and the triple $(k_e,k_h,k_f)$. By the way they are defined, they satisfy $k+k_e+k_h+2k_f=n$, where $n$ is the number of degrees of freedom of the integrable system.

The following is a result of Eliasson \cite{Eliasson90} and Miranda and Zung (\cite{MirandaThesis}, \cite{MirandaCentrEur}, \cite{MirandaZung}).

\begin{thm}[Smooth local linearization]
Given an smooth integrable Hamiltonian system with $n$ degrees of freedom on a symplectic manifold $(M^{2n},\omega)$, the Liouville foliation in a neighborhood of a non-degenerate singular point of rank $k$ and Williamson type $(k_e,k_h,k_f)$ is locally symplectomorphic to the model Liouville foliation, which is the foliation defined by the basis functions of Theorem \ref{thm:WilliamsonCartanSubalgebra} plus "coordinate functions" $f_i=x_i$ for $i=k_e+k_h+2j+1$ to $n$.
\label{thm:locallinearization}
\end{thm}

\begin{rem}
The theorem states the existence of a semilocal symplectomorphism between foliations with a non degenerate singularity of rank $k$ and the same parameters $(k_e,k_h,k_f)$. One could think that functions are also preserved via a symplectomorphism, but it is not possible to guarantee this statement when $h_k \neq 0$ as one can add up analytically flat terms on different connected components (see counterexample in \cite{MirandaThesis}).  In general one needs more information about the topology of the leaf to conclude (see Figure \ref{fig:twistedhyperbolic}).
\end{rem}

\begin{rem}
Because of Theorem \ref{thm:locallinearization}, if one considers the Taylor expansions of $F=(f_1, \dots,f_n)$ at the non-degenerate singular point in a canonical coordinate system and removes all terms except for linear and quadratic, the functions obtained remain commuting and define a Liouville foliation that can be considered as the \textit{linearization} of the initial foliation $\mathcal F$ given by $f_1, \dots,f_n$, to which it is symplectomorphic.
\end{rem}

The description of non-degenerate singularities at the semilocal level is completed with the following two results.

\begin{thm}[Model in a covering]
The manifold can be represented, locally at a non-degenerate singularity of rank $k$ and Williamson type $(k_e,k_h,k_f)$, as the direct product
$$M^{\text{reg}} \times \overset{k}\cdots \times M^{\text{reg}} \times M^{\text{ell}} \times \overset{k_e}\cdots \times M^{\text{ell}} \times M^{\text{hyp}} \times \overset{k_h}\cdots \times M^{\text{hyp}} \times M^{\text{foc}} \times \overset{k_f}\cdots \times M^{\text{foc}}$$
Where:
\begin{itemize}
    \item $M^{\text{reg}}$ is a "regular block", given by $$f=x,$$
    \item $M^{\text{ell}}$ is an "elliptic block", representing the elliptic singularity given by $$f=x^2 + y^2,$$
    \item $M^{\text{hyp}}$ is an "hyperbolic block", representing the hyperbolic singularity given by $$f=xy,$$
    \item $M^{\text{foc}}$ is a "focus-focus block", representing the focus-focus singularity given by
    $$\begin{cases}f_1 = x_1 y_{2}- x_{2} y_1 \\
    f_{2} = x_1 y_1 + x_{2} y_{2}\end{cases}.$$
\end{itemize}
For the first three types of blocks the symplectic form is $\omega = dx \wedge dy$, while for the focus-focus block it is $\omega = dx_1 \wedge dy_1 + dx_2 \wedge dy_2$.
\label{thm:directproduct}
\end{thm}

In the case of a smooth system (defined by a smooth moment map), a similar result was proved and described by Miranda and Zung in \cite{MirandaZung}. It summarizes some previously results proved independently and fixes the case where there are hyperbolic components ($k_h\neq 0$), because in this case the result is slightly different and it has to be taken the semidirect product in the decomposition. As opposite to the case where there are only elliptic and focus-focus singularities, in which the base of the fibration of the neighbourhood is an open disk, if there are hyperbolic components the topology of the fiber can become complicated. The reason is essentially that for the smooth case a level set of the form $\{x_iy_i=\epsilon\}$ is not connected but consists of two components.

\begin{thm}[Miranda-Zung]
Let $V = D^k \times \bbT^k \times D^{2(n-k)}$ with coordinates $(p_1,...,p_k)$ for $D^k$, $(q_1 (mod 1),...,q_k (mod 1))$ for $\bbT^k$, and $(x_1,y_1,...,x_{n-k},y_{n-k})$ for $D^{2(n-k)}$ be a symplectic manifold with the standard symplectic form $\sum dp_i
\wedge dq_i + \sum dx_j \wedge dy_j$. Let $F$ be the moment map corresponding to a singularity of rank $k$ with Williamson type $(k_e,k_h,k_f)$. There exists a finite group $\Gamma$,
a linear system on the symplectic manifold $V/\Gamma$
and a smooth Lagrangian-fibration-preserving symplectomorphism
$\phi$ from a neighborhood of $O$ into $V/\Gamma$, which sends $O$
to the torus $\{p_i=x_i=y_i = 0\}$. The smooth symplectomorphism
$\phi$ can be chosen so that via $\phi$, the system-preserving
action of a compact group $G$ near $O$ becomes a linear
system-preserving action of $G$ on $V/\Gamma$. If the moment map
$F$ is real analytic and the action of $G$ near $O$ is
analytic, then the symplectomorphism $\phi$ can also be chosen to
be real analytic. If the system depends smoothly (resp.,
analytically) on a local parameter (i.e. we have a local family of
systems), then $\phi$ can also be chosen to depend smoothly
(resp., analytically) on that parameter.
\end{thm}

In this case, the so-called \textit{twisted hyperbolic} component can arise (see Figure \ref{fig:twistedhyperbolic}), and the group of all linear moment maps preserving symplectomorphisms of the linear direct model of
Williamson type $(k_e, k_h, k_f)$ is isomorphic to
$$\bbT^k \times \bbT^{k_e} \times  (\bbR \times \bbZ/2\bbZ)^{k_h} \times (\bbR \times \bbT^1)^{k_f}.$$

\begin{figure}[ht]
\begin{tikzpicture}[scale=1]
\coordinate (A) at (4.5,2.3);
\coordinate (B) at (2.5,3.5);
\path[->,>=stealth',ultra thick,color=blue!60!green,dashed]
    (A) edge[bend right=90] (B);
\begin{axis}[domain=-4:4, samples=1000, restrict y to domain=-4:4, axis x line=center, axis y line=center, tick style={draw=none}, xticklabels={,,}, yticklabels={,,}]
\addplot [color=red,thick]  {0.2/x};
\addplot [color=red!70!black,thick]    {1/x};
\addplot [color=red!40!black,thick]  {2.5/x};
\addplot [color=blue,thick]  {-0.2/x};
\addplot [color=blue!60!black,thick]    {-1/x};
\addplot [color=blue!20!black,thick]  {-2.5/x};
\end{axis}
\node at (A)[circle,fill,inner sep=1.5pt]{};
\node at (B)[circle,fill,inner sep=1.5pt]{};
\coordinate (start) at (6.85,2.65);
\coordinate (end) at (3.2,5.7);
\coordinate (start2) at (0,3.05);
\coordinate (end2) at (3.65,0);
\draw[->,>=stealth',ultra thick,color=blue!60!green]
    (B) .. controls (2,3) and (3,2) .. (A);
\end{tikzpicture}
\caption{In the neighbourhood of an orbit of rank $1$ and Williamson type $(0,1,0)$, the return map corresponding to the flow of circle action can give rise to two different behaviours. After one turn, the point can return to itself or it can return to its "opposite" branch (twisted hyperbolic case), and this defines a $\bbZ/2\bbZ$ action. The twisted hyperbolic case is described in this picture.}
\label{fig:twistedhyperbolic}
\end{figure}
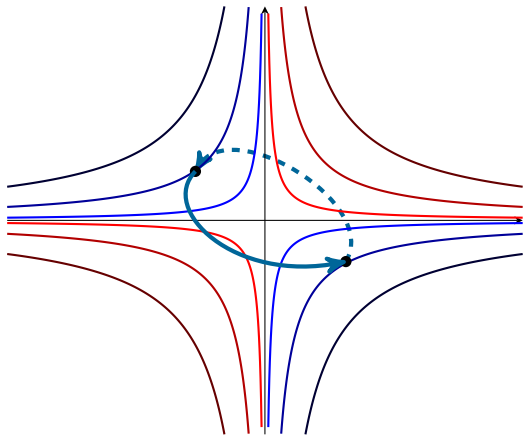

To end this section, we recall a related result which highlights the importance of considering the symplectomorphism at the level of the Lagrangian fibration induced by the Hamiltonian vector fields of the integrable system. Assume that $(M,\omega)$ is a symplectic manifold with a non-degenerate singularity of Williamson type $(k_e,k_h,k_f)$. Assume that the foliation $\mathcal{F}$ at the singularity is the linear foliation defined by $\mathcal F=\langle X_1,\dots, X_n\rangle$, where the vector fields $X_i$ are the linear Hamiltonian vector fields corresponding to the basis functions of Theorem \ref{thm:WilliamsonCartanSubalgebra}. Namely, $X_i$ are the vector fields induced by $\iota_{X_i}\omega=-df_i$, that is:
\begin{itemize}
    \item $X_i = -y_i\frbd{x_i} + x_i\frbd{y_i}$ \quad for elliptic components,
    \item $X_i = -x_i\frbd{x_i} + y_i\frbd{y_i}$ \quad for hyperbolic components,
    \item $X_i = -x_i\frbd{x_i} + y_i\frbd{y_i} -x_{i+1}\frbd{x_{i+1}} + y_{i+1}\frbd{y_{i+1}} \quad \text{and}\\
X_{i+1} = x_{i+1}\frbd{x_i} + y_{i+1}\frbd{y_i} -x_i\frbd{x_{i+1}} -y_i\frbd{y_{i+1}}$
\quad for focus-focus components.
\end{itemize}
Then, the following theorem holds.

\begin{thm}{\cite{MirandaThesis}}
Let $\omega$ be a symplectic form defined in a neighbourhood of the singularity at $p$ for which the foliation $\mathcal F$ is Lagrangian. Then, there exists a local diffeomorphism $\phi:(U,p)\longrightarrow (\phi(U),p)$ such that $\phi$ preserves the foliation and $\phi^*(\sum_i dx_i\wedge dy_i)=\omega$, where $x_i,y_i$ are local coordinates on $(\phi(U),p)$.
\label{thm:nbhdnondegsing}
\end{thm}

For completely elliptic singularities (of rank $0$ and Williamson type $(k_e,0,0)$) Theorem \ref{thm:nbhdnondegsing} was proved by Eliasson \cite{Eliasson90}.
When $h_e\neq 0$, the foliation given by the hyperbolic components is preserved but the components of the moment map are not necessarily preserved (for more details see \cite{MirandaThesis}).

\begin{rem}\label{rem:mather}
All the theorems above can be understood in the language of Mather \cite{mather} and Thom \cite{thom} as stability theorems for the integrable systems.

In \cite{evasan} we studied the infinitesimal stability of integrable systems. The theorem above can be seen,in the spirit of Mather, as an \emph{infinitesimal stability implies stability} theorem.
\end{rem}

\section{Equivalence of close lifted actions}
\label{sec:RigidityCotLift}
We state and prove some results on symplectic equivalence of lifted close actions of a compact group on a compact manifold. We start proving a proposition on the equivalence at the level of cotangent lift given equivalence at the base. It is clear that if two symplectic actions are close, so are their fundamental vector fields. In Proposition \ref{prop:equivactionsandmomentmaps} we prove that if two actions are $C^1$-equivalent, so are their cotangent lifts, and we define explicitly the diffeomorphism that conjugates them. With the same idea, and since any cotangent lifted action is Hamiltonian, we prove that if two actions are $C^1$-equivalent, then the moment maps induced by their cotangent lifts are also equivalent.

\begin{prop}
Let $G$ be a Lie group and let $M$ be a smooth manifold. Let $\rho_1,\rho_2: G \times M \longrightarrow M$ be two actions which are $C^1$-equivalent via a conjugation through a diffeomorphism $\varphi$. Let $\hat{\rho}_1,\hat{\rho}_2$ be the cotangent lifts of $\rho_1,\rho_2$, respectively. Then, $\hat{\rho}_1$ and $\hat{\rho}_2$ are $C^1$-equivalent via the conjugation through $\hat{\varphi}$. The moment maps induced by $\hat{\rho}_1,\hat{\rho}_2$, denoted respectively by $\mu_1,\mu_2:T^*M\longrightarrow \mfg^*$, are equivalent via the conjugation through $\hat{\varphi}$ as $\mu_2=\mu_1\circ \hat{\varphi}$.
\label{prop:equivactionsandmomentmaps}
\end{prop}

\begin{proof}
Assume $\rho_1,\rho_2: G \times M \longrightarrow M$ are two $C^1$-equivalent Lie group actions. Let $\varphi$ be the $C^1$-diffeomorphism conjugating the two actions, i.e, let $\varphi$ be a diffeomorphism such that $\rho_1\circ\varphi=\varphi\circ\rho_2$. Differentiating both sides, the following equality is obtained: 
$$d\rho_{1,\varphi(q)}\circ d\varphi_q=d\varphi_{\rho_2(q)}\circ d\rho_{2,q}.$$

Transposing and inverting the latter equality on both sides, one arrives to the following relation:
$$((d\rho_{1,\varphi(q)})^*)^{-1}\circ ((d\varphi_q)^*)^{-1}(p)=
((d\varphi_{\rho_2(q)})^*)^{-1}\circ ((d\rho_{2,q})^*)^{-1}(p),$$ which shows that $((d\varphi)^*)^{-1}$ is exactly the conjugation between $((d\rho_{1,\varphi(q)})^*)^{-1}$ and $((d\rho_{2,q})^*)^{-1}$.

We define now $\hat{\varphi}(q,p):=(\varphi(q),((d\varphi_q)^*)^{-1}(p))$, which is a diffeomorphism and can be thought as the \textit{cotangent lift of $\varphi$}. Consider the cotangent lift of the actions $\rho_1$ and $\rho_2$, i.e. $\hat{\rho}_1$ and $\hat{\rho}_2$. By definition, $\hat{\rho}_i(q,p)=(\rho_i(q),((d{\rho}_{i,q})^*)^{-1}(p))$. Then, it is clear that $\hat{\rho}_1\circ\hat{\varphi}=\hat{\varphi}\circ\hat{\rho}_2$, and we conclude that the cotangent lifts of the actions are equivalent on the cotangent bundle via conjugation by $\hat{\varphi}$, which is precisely the cotangent lift of the diffeomorphism $\varphi$ that conjugates $\rho_1$ and $\rho_2$ on the base.

The cotangent lift of the action $\hat{\rho}_i$ is a Hamiltonian action with moment map $\mu_i:T^*M \longmapsto \mfg^*$ given by
$$\langle \mu_i(p),X \rangle := \langle \lambda_p ,X^\#|_{p} \rangle =\langle p,X^\#|_{\pi(p)} \rangle,$$
where  $p\in T^*M, X\in\mfg$, $X^\#$ is the fundamental vector field of $X$ generated by the $\hat{\rho}_i$ action and $\lambda$ is the Liouville 1-form on $T^*M$.

The diffeomorphism $\hat{\varphi}$ is defined by
$$\hat{\varphi}(q,p):=(\varphi(q),((d\varphi_q)^*)^{-1}(p))$$ and satisfies $\hat{\rho}_1\circ\hat{\varphi}=\hat{\varphi}\circ\hat{\rho}_2$. On the other hand, by Lemma \ref{lem:cotangentliftpres1form} the Liouville one-form is invariant under the lifted actions, i.e. $\hat \rho_i^* \lambda = \lambda$ for $i=1,2$ and it is also invariant under the diffeomorphism $\hat\varphi$ by Lemma \ref{lem:cotangentliftpres1form}.

Through the following computation:
\begin{align*}
    \langle \mu_2(p),X \rangle &= \langle \lambda_p,X_2^\#|_p \rangle=\\
    &= \langle \lambda_p,\frac{d}{dt}\left(\hat\rho_2(\exp(-tX),p) \right)|_{t=0} \rangle=\\
    &= \langle \lambda_p,\frac{d}{dt}\left(\hat\varphi^{-1}(\hat\rho_1(\exp(-tX),\hat\varphi(p))) \right)|_{t=0}\rangle=\\
    &= \langle \lambda_{\hat\varphi(p)},\frac{d}{dt}\left(\hat\rho_1(\exp(-tX),\hat\varphi(p)) \right)|_{t=0} \rangle=\\
    &= \langle \lambda_{\hat\varphi(p)},X_1^\#|_{\hat\varphi(p)} \rangle =\\
    &= \langle \mu_1(\hat\varphi(p)),X \rangle = \langle \mu_1\circ \hat\varphi(p),X \rangle,
\end{align*}
where we have used that $\hat{\varphi}^{-1}\circ\hat{\rho}_1\circ\hat{\varphi}=\hat{\rho}_2$. Observe that the fundamental vector fields and the actions are $\hat\varphi$-related. If one of the fundamental vector fields is Hamiltonian in the $\xi$ direction (the one given by $\mu_1$), so is the second (the one given by $\mu_1\circ\hat\varphi$). We conclude that the moment maps are equivalent.

\end{proof}

Now we prove a theorem that can be thought as the cotangent lifted version of Theorem \ref{th:palaissymplectic}.

\begin{thm}
Let $G$ be a compact Lie group and $M$ a compact smooth manifold. Let $\rho_1,\rho_2: G \times M \longrightarrow M$ be two actions which are $C^1$-close. Let $\hat{\rho}_1,\hat{\rho}_2: G \times (T^*M,\omega) \longrightarrow (T^*M,\omega)$ be the cotangent lifts of $\rho_1,\rho_2$, respectively. Then, there exists a symplectomorphism that conjugates $\hat{\rho}_1$ and $\hat{\rho}_2$, thus making them equivalent.
\label{thm:palaiscotangentlift}
\end{thm}

\begin{rem}
Notice that the actions have to be $C^1$-close. Compared with the symplectic version of Palais rigidity Theorem (Theorem \ref{th:palaissymplectic}), where they have to be $C^2$-close, one degree of differentiability is gained here.
\end{rem}

\begin{proof}
Let $G$ be a compact Lie group and $M$ a compact smooth manifold. Let $\rho_1,\rho_2:G\times M\longrightarrow M$ be two actions and assume that they are $C^1$-close. By Theorem \ref{th:palais}, there exists a diffeomorphism $\varphi$ that conjugates $\rho_1$ and $\rho_2$.

Consider $\hat{\rho}_1,\hat{\rho}_2: G \times (T^*M,\omega) \longrightarrow (T^*M,\omega)$, the cotangent lifts of $\rho_1$ and $\rho_2$, respectively. By Proposition \ref{prop:equivactionsandmomentmaps}, the diffeomorphism $\hat{\varphi}$ conjugates $\hat{\rho}_1$ and $\hat{\rho}_2$. To prove that the actions $\hat{\rho}_1$ and $\hat{\rho}_2$ are not only equivalent, but symplectically equivalent, we need to check that $\hat{\varphi}$ preserves the symplectic form. By Lemma \ref{lem:cotangentliftpres1form}, it preserves the canonical $1$-form $\lambda$ of $T^*M$ and, hence, it preserves the symplectic form $\omega$.
\end{proof}

\section{Application to integrable systems with non-degenerate singularities}
\label{sec:ApplIntSys}

Results of the previous section, namely  shows a natural way of applying the result of rigidity of the lifted actions to the category of Hamiltonian systems. Theorem \ref{thm:palaiscotangentlift} guarantees, for instance, that the compact orbits of two $C^1$-close integrable systems on a symplectic manifold are equivalent at the level of the cotangent lift.

An immediate corollary of Palais rigidity Theorem is the following. Consider two integrable systems in a compact symplectic manifold $(M,\omega)$ given by $F=(f_1,\dots,f_n)$ and $\hat{F}=(\hat{f}_1,\dots,\hat{f}_n)$, respectively. Let $X_{1},\dots,X_{n}$ and $\hat X_{1},\dots,\hat X_{n}$ be the corresponding associated vector fields (those induced by $\iota_{X_i}\omega=-df_i$). If, for each $i=1,\dots,n$, the flow $\psi_i$ of $X_{i}$ is close to the flow $\hat\psi_i$ of $\hat X_{i}$, and all of them are actions of a compact group (case of toric manifolds), then the two integrable systems are equivalent, i.e., it exists a diffeomorphism $\varphi$ that conjugates $F$ and $\hat{F}$.
This equivalence can even be pictured in terms of the Delzant theorem looking at the corresponding Delzant polytopes \cite{delzant}.

In the same direction, a straightforward consequence of Theorem \ref{thm:palaiscotangentlift} at the semilocal level in a neighbourhood of a compact orbit is the following.

\begin{thm}
Let $F=(f_1,\dots,f_n):(M^{2n},\omega)\rightarrow\bbR^n$ and $\hat{F}=(\hat{f}_1,\dots,\hat{f}_n):({M}^{2n},\omega)\rightarrow\bbR^n$ be two smooth maps defining two integrable systems. Suppose that the singularities of $F$ and $\hat F$ are non-degenerate and a combination of only regular and elliptic components (with compact orbits) i.e., that each singularity of rank $k\neq n$ has Williamson type $(n-k,0,0)$. Assume that, for all $1\leq i \leq n$, $f_i$ and $\hat{f}_i$ are $C^2$-close. Then, for each $c \in \text{Im}(F)\subset \bbR^n$:
\begin{enumerate}
    \item there exists $\hat c \in \text{Im}(\hat F)\subset \bbR^n$ that is close to $c$, and
    \item there exists a symplectomorphism $\phi_c$ that makes the neighbourhoods of the leaves $\Lambda_c=F^{-1}(c)$ and $\hat\Lambda_{\hat c}=\hat F^{-1}(\hat c)$ equivalent. Namely, there exists $\phi_c$ defined in a neighbourhood of $\Lambda_c$ such that $\phi_c\circ F = \hat{F} \circ \phi_c$ and $\phi_c^*(\omega)=\omega$.
\end{enumerate}
\label{thm:PalaisHamSyst}
\end{thm}

\begin{rem}
Observe that for elliptic and regular components the connected components of the leaves equal the orbits.
\end{rem}
\begin{proof}
By closeness between $F$ and $\hat F$, for each $c \in \text{Im}(F)\subset \bbR^n$ there exists $\hat c \in \text{Im}(\hat F)\subset \bbR^n$ that is close to $c$ and such that $\hat c$ is a singular value of $\hat F$ if and only if $c$ is a singular value of $F$. Closeness between $F$ and $\hat F$ (together with non-degeneracy) guarantees that the number of elliptic components at the singularity $x\in F^{-1}(c)$ is the same as the number of elliptic components at $y\in \hat F^{-1}(\hat c)$.

Now, in view of Theorem \ref{thm:directproduct}, and since in this case the singularities are the product of only regular and elliptic type, if we prove the existence of the symplectomorphism for the case of a regular value and for the case of a complete elliptic singularity we will be finished.

If $c$ is a regular value of $F$, by the Arnold-Liouville-Mineur Theorem the neighbourhood of the leaf $\Lambda_c$ is diffeomorphic to the cotangent bundle of the Liouville torus. The same applies to the the neighbourhood of the leaf $\hat\Lambda_{\hat c}$. The action on $T^*\bbT^n$ is the cotangent lift of a compact torus action and then, by Theorem \ref{thm:palaiscotangentlift}, there exists a symplectomorphism $\phi_c$ conjugating $F$ and $\hat F$ on the respective leaf neighbourhoods.

Now suppose $c$ is a non-degenerate singular value of $F$ and $x\in F^{-1}(c)$ is a completely elliptic singularity. Consider the action given by the joint flow, which in this case is locally free and has a unique fixed point, the singularity $x$ . By means of the joint flow we identify the action as a torus action (see \cite{MirandaZung}) and we can apply Theorem \ref{th:palaissymplectic} to obtain rigidity between a neighbourhood of $\Lambda_c$ and $\hat\Lambda_{\hat c}$.
\end{proof}

\begin{rem}
In the case of a regular point, another way of proving symplectic rigidity is using the normal form of the moment map, since there is only one local model, which is the one given by the Arnold-Liouville-Mineur Theorem. 
\end{rem}

\begin{rem}
We do not require that the Williamson type of the non-degenerate singularities of $F$ and $\hat F$ is the same, only that they both are combination of regular and elliptic type (in both cases the orbits coincide with the leafs). Notice that if $\hat F$ is close enough to $F$, the elliptic components of a singularity of $F$ will remain elliptic in the associated singularity of $\hat F$, and the regular components can not become neither hyperbolic nor focus-focus, so compactness of actions and, hence, rigidity, is guaranteed without having to impose the same Williamson type.
\end{rem}

These consequences do not go beyond results that are already known concerning rigidity of integrable systems. In fact, they can be considered special cases of the Arnold-Lioville-Mineur Theorem, since it gives a unique normal form for neighbourhoods of regular points of integrable systems. Nevertheless, Theorem \ref{thm:palaiscotangentlift} can be used in the same context of integrable systems to prove a slightly more ambitious result.

\section{Application to $S^1$-invariant degenerate singularities}
\label{sec:ApplDegSing}
Consider the following example of a very simple integrable system.

\begin{examp}
Let $f=(x^2+y^2)^k$, with $k\geq2$, be the moment map of an the integrable system in $(\bbR^2,\omega_{st}=dx\wedge dy)$. It is a completely solvable system, it has an isolated degenerate singularity at the origin, the flows of the Hamiltonian vector field lie in concentric circles, and the singularity is a stable center. Since it is a degenerate singularity, we can not apply directly   normal form theorems. Nevertheless, we know that, the system is invariant with respect to the $S^1$ action and therefore we can use another system (which is non-degenerate) associated to the circle action for which  there exists a normal form, which in fact is $x^2+y^2$ and corresponds to an elliptic singularity.
\label{examp:r^2k}
\end{examp}

In order to  conclude we need a normal form result for circle actions. We first recall the general symplectic slice theorem and then apply it in the case of a fixed point of a circle action.
\begin{thm}[Guillemin-Sternberg \cite{guilleminsternbergnormalform}, Marle \cite{marle1}]\label{thm:symplecticslicetheorem}
Let $(M, \omega, G)$ be a symplectic manifold together with a Hamiltonian group action. Let $z$ be a point in $M$ such that $\mathcal{O}_z$ is contained in the zero level set of the momentum map. Denote $G_z$ the isotropy group and $\mathcal{O}_z$ the orbit of $z$. There is a $G$-equivariant symplectomorphism from a neighbourhood of the zero section of the bundle $T^*G \times_{G_z} V_z$ equipped with the above symplectic model to a neighbourhood of the orbit $\mathcal{O}_z$.
\end{thm}

Recall from Bochner's linearization theorem that in a neighbourhood of a fixed point of an action we can always linearize the group action.

Applying the theorem \ref{thm:symplecticslicetheorem} above to the circle action case with a fixed point and applying Bochner's theorem we obtain the classical Marle-Guillemin-Sternberg which gives a local normal form for the moment map of circle actions in a neighbourhood of a fixed point of the action.

\begin{thm}[\cite{marle1,marle2, guilleminsternbergnormalform}]\label{thm:mgs}
Let $(M^{2n},\omega)$ be a symmplectic manifold endowed with an $S^1$-Hamiltonian action and let $p$ be a fixed point for this action. Then there exist local coordinates $(x_1,y_1,\dots, x_n,y_n)$ such that $\omega=\sum_{i=1}^n dx_i\wedge dy_i$
and $\mu(x)=\sum_{i=1}^n c_i(x_i^2+y_i^2)$.

\end{thm}

\begin{rem}
The constants $c_j$ can be interpreted as \emph{weights} of the circle action.
\end{rem}

The last conclusion of the example is summarized the following Lemma, which is an easy consequence of the Guillemin-Marle-Sternberg Theorem.

\begin{lem}
Consider a 2-dimensional integrable system which has an $S^1$-invariant degenerate singularity. Then, locally it is function of the quadratic normal form of elliptic type.
\label{lem:degsingS1invariant}
\end{lem}

\begin{proof}
By Guillemin-Marle-Sternberg Theorem \ref{thm:mgs}, the moment map of an $S^1$-action with a fixed point is a sum of squares in its normal form. Since it is a 2-dimensional system and because of Noether's theorem, one can take coordinates $x,y$ in a neighbourhood of the singularity in such a way that the moment map can be written as $$f=\phi(x^2+y^2).$$
\end{proof}

Consider now $\bbR^4$ with coordinates ($x_1,y_1,x_2,y_2$) and with the standard symplectic form $\omega_{st}=dx_1\wedge dy_1+dx_2\wedge dy_2$. Consider the three following Hamiltonian functions:
\begin{align}
    F=&(f_1,f_2)=\left(x_1^2+y_1^2,x_2^2+y_2^2\right)\\
    G=&(g_1,g_2)=\left((x_1^2+y_1^2)^2,x_2^2+y_2^2\right)\\
    H=&(h_1,h_2)=\left((x_1^2+y_1^2)^2,(x_2^2+y_2^2)^2\right)
\end{align}

The three integrable systems have an isolated singularity at the origin, but only in the system given by $F$ it is non-degenerate. By the way, this system is the model of the two uncoupled harmonic oscillators and its level sets are 2-dimensional invariant tori. The same level sets appear in the other two systems, but Theorem \ref{thm:PalaisHamSyst} can be directly applied to state rigidity only in the first system, since it has a single non-degenerate singularity of elliptic type (it is already in normal form), while in the others the singularity is degenerate. Observing the system given by $G$, though, one can see that the second component produces an invariant $S^1$ action. This $S^1$ symmetry allows for a symplectic reduction of the system, making it decrease from a 4-dimensional to a 2-dimensional. In this new system, the singularity is still degenerate, but following the idea in Example \ref{examp:r^2k} one can find its non-degenerate elliptic normal form. Then, Theorem \ref{thm:PalaisHamSyst} can be applied to obtain rigidity of the reduced system, understanding rigidity as equivalence of close systems. It is not difficult to see that, then, the original system is also rigid.

In view of this procedure, we have the following result.

\begin{thm}
Consider an integrable system in a symplectic manifold $(M,\omega)$ given by $F=(f_1,\dots,f_n)$. Suppose that if $p\in M$ is a singularity of $F$, it is isolated, there are no other singularities in its $F$-level set, and it is:
\begin{itemize}
    \item either non-degenerate of regular or elliptic type, or
    \item degenerate of the following type: $f_1,\dots,f_{n-1}$ have a non-degenerate singularity of elliptic type at $p$, $f_n$ has a degenerate singularity at $p$ and $f_n$ is $S^1$-invariant.
\end{itemize}
Then the system is rigid at the neighbourhood of each compact leaf $\Lambda_c=F^{-1}(c)\subset M$.
\label{thm:rigiditysystemwithdegensing}
\end{thm}

\begin{proof}
In all the regular leaves or in the leaves containing non-degenerate singularities, Theorem \ref{thm:PalaisHamSyst} already gives rigidity. At any singular leaf containing a degenerate singularity, there exist $(n-1)$ $S^1$-invariant actions  that commute so we can perform a series of $(n-1)$ symplectic reductions  successively to reduce the system to a 2-dimensional system, which has a degenerate singularity corresponding to the singularity of $f_n$. At this point, the moment map of the reduced integrable system still gives an $S^1$-invariant action which has a moment map $\overline{f_n}$ and because of Theorem \ref{thm:mgs} the function $\overline{f_n}$ can be put in the quadratic normal form corresponding to the elliptic singularity. Then, again by Theorem \ref{thm:PalaisHamSyst}, the system associated to $\overline{f_n}$ is rigid at the neighbourhood of the leaf. Because of by Lemma \ref{lem:degsingS1invariant}  the function $f_n$ is a smooth function of $f_n=H(\overline{f_n})$ of $\overline{f_n}$ and thus rigidity also holds for $f_n$  and by reconstruction from the initial
integrable system $(f_1,\dots, f_n)$ in a neighbourhood of a compact leaf.
\end{proof}

Theorem \ref{thm:rigiditysystemwithdegensing} states semiglobal rigidity in the very particular case of systems with degenerate singularities that are non-degenerate in $(n-1)$ components of the moment map and have an $S^1$-invariant action in the degenerate component.

\begin{rem} From a dynamical point of view, the results included in this paper can be understood as a weak KAM theorem where Hamiltonian perturbations occur in the subclass of integrable systems.
It would be interesting to explore the weak analogues for focus-focus singularities which can be seen as a cotangent lift as shown in example \ref{examp:focusfocuscotangentlift}.Those singularities are infinitesimally stable \cite{evasan} and stable (see for instance \cite{sanfocusfocus}, \cite{chaperonfocusfocus}) however it is not possible to follow the guidelines above due to the lack of compactness of the group $S^1\times \mathbb R$.
\end{rem}

\bibliographystyle{alpha}
\bibliography{RigidityMirMiranda}

\end{document}